\documentclass[11pt]{article}

\usepackage{sn-preamble}
\usepackage[letterpaper, left=1in, right=1in, top=1in, bottom=1in]{geometry}
\usepackage{authblk}
\usepackage{microtype}

\newcommand{\vv}[1]{\bv^{(#1)}}

\allowdisplaybreaks

\title{On the Maximal Gaussian Perimeter of Convex Sets, Revisited}
\author[]{Shivam Nadimpalli}
\author[]{Caleb Pascale}
\affil[]{Massachusetts Institute of Technology} 
\affil[]{\texttt{\{shivamn,\,psvita\}@mit.edu}}

\date{\today}

\begin{document}


\maketitle

\begin{abstract}
    We revisit Nazarov's construction~\cite{Nazarov:03} of a convex set with nearly maximal Gaussian surface area and give an alternate analysis based on the notion of \emph{convex influence} introduced in~\cite{DNS21itcs,DNS22}. 
\end{abstract}


\section{Introduction}
\label{sec:intro}

The \emph{Gaussian surface area} of a (Borel) set $K \sse \R^n$ is defined as   
\[
    \GSA(K) := \lim_{\delta \to 0} \frac{\Vol(K_t \setminus K)}{t}\,,
\]
where $K_t$ is the $t$-enlargement of $K$ (that is, $K_t = \{x \in \R^n : \exists y\in K~\text{s.t.}~\|x-y\|\leq t \}$) and $\Vol(K)$ denotes the measure of $K$ under the $n$-dimensional standard Gaussian distribution $N(0,I_n)$. 
It is a fundamental complexity measure in high-dimensional geometry, arising naturally in the context of noise stability and isoperimetry~\cite{Borell:75,sudakov1978extremal}. 
It has also found applications to problems in probability theory~\cite{bentkus1986dependence,Bentkus:03,raivc2019multivariate} and theoretical computer science~\cite{KOS:08,Kane:10,DMN19,DNS23-polytope}. 

For convex sets (and more generally, sets with smooth boundary), the above definition  coincides with the following (see~\cite{Nazarov:03}): 

\begin{definition} \label{def:GSA}
    If $K \sse \R^n$ is convex, then 
    \[
        \GSA(K) = \int_{\partial K} \phi_n(x) \, d\sigma(x)\,,
    \]
    where $d\sigma(\cdot)$ denotes the surface measure in $\R^n$ and $\phi_n(x) = (2\pi)^{-n/2} e^{-\|x\|^2/2}$ is the density of the $n$-dimensional standard Gaussian measure. 
\end{definition}

Motivated by problems in high-dimensional probability, Ball~\cite{Ball:93} considered the following reverse isoperimetric problem for the Gaussian measure: how large can $\GSA(K)$ be for a convex set $K \sse \R^n$? 
It will be convenient to write 
\[
    \Gamma(n) := 
    \sup\cbra{\GSA(K) : \text{convex}~K\sse\R^n}\,.
\]
Ball established the following uniform bound on the Gaussian surface area of convex sets: 

\begin{theorem}[Theorem~4 of~\cite{Ball:93}] \label{thm:ball}
    We have $\Gamma(n) \leq 4n^{1/4}$. 
\end{theorem}

This order of growth is sharp: Nazarov~\cite{Nazarov:03} showed the existence of a convex set $K \sse \R^n$ with surface area on the order of $\Omega(n^{1/4})$ for $n$ sufficiently large.   
Nazarov's construction has since found applications in learning theory~\cite{KOS:08}, polyhedral approximation~\cite{DNS23-polytope}, and property testing~\cite{CDNSW23inprep}. 


\begin{theorem}[\cite{Nazarov:03}] \label{thm:naz}
    We have 
    \[
        \liminf_{n\to\infty} \frac{\Gamma(n)}{n^{1/4}} \geq e^{-5/4}\,. 
    \]
\end{theorem}
Nazarov~\cite{Nazarov:03} lamented that his proof of~\Cref{thm:naz} (and of his sharpening of~\Cref{thm:ball}; see below) was ``a pretty boring and technical computation'' and speculated that ``there should exist some simple and elegant way leading to the result,''  even encouraging the reader to ``stop reading and (try to) prove the theorem themselves.'' 
This note attempts to give such a proof. 
In particular, our argument avoids the heavier analytic machinery (such as the Laplace asymptotic formula) and technical estimates of Nazarov's original proof. 
Our key idea is to analyze an isoperimetric quantity closely related to Gaussian surface area, namely  \emph{convex influence} (introduced in~\cite{DNS21itcs,DNS22} and recalled below in~\Cref{subsec:convex-influence}), which is considerably easier to calculate.  

We note that a gap remains between the best known upper and lower bounds for $\Gamma(n)$. 
Nazarov~\cite{Nazarov:03} himself sharpened Ball's bound (\Cref{thm:ball}), and a subsequent refinement by Rai\v{c}~\cite{raivc2019multivariate} obtained the following estimate, which, to the best of our knowledge, is the current state of the art:
\[
    \Gamma(n) \leq \sqrt{\frac{2}{\pi}} + 0.59(n^{1/4} - 1)\,. 
\]
\Cref{thm:naz} continues to be the best known lower bound on $\Gamma(n)$; note that $e^{-5/4} \approx 0.2865$.  

\

\noindent \textbf{Organization.}~We recall useful preliminaries in~\Cref{sec:prelims} and present our proof of~\Cref{thm:naz} in \Cref{sec:proof}. 

\section{Preliminaries}
\label{sec:prelims}

We write $[n] := \{1, \dots, n\}$. 
We use boldfaced letters (such as $\bx$, $\bv$, and $\bK$) to denote random variables (which may be real-valued, vector-valued, or set-valued; the intended type will be clear from the context. 
We write $\bx \sim \calD$ to indicate that the random variable $\bx$ is distributed according to probability distribution $\calD$. 
Finally, $\S^{n-1} = \{x \in \R^n : \|x\| = 1\}$ will denote the unit sphere in $\R^n$. 

\subsection{Distributions and Tail Bounds}
\label{subsec:prob}

We write $N(0, I_n)$ for the $n$-dimensional standard Gaussian distribution over $\R^n$ with density $\phi_n(\cdot)$. 
When the dimension $n$ is clear from context, we will sometimes write $\phi$ instead of $\phi_n$. 
The Gaussian measure of a (measurable) set $K \sse \R^n$ will be denoted by $\Vol(K)$, i.e., 
\[
	\Vol(K) := \Prx_{\bx\sim N(0, I_n)}\sbra{\bx \in K}\,.
\]
We recall a standard concentration bound for univariate Gaussian random variables:

\begin{proposition}[Proposition 2.1.2 of~\cite{vershynin2018high}] \label{prop:gaussian-tails}
	Let $\bx \sim N(0, 1)$. 
	Then for all $t > 0$, we have 
	\[
		\pbra{\frac{1}{t} - \frac{1}{t^3}}\cdot\frac{1}{\sqrt{2\pi}}e^{-t^2/2} \leq \Pr\sbra{\bg\geq t} \leq \frac{1}{t}\cdot\frac{1}{\sqrt{2\pi}}e^{-t^2/2}\,.
	\]
\end{proposition} 

We will also require the following concentration bound on the norm of a Gaussian random vector: 

\begin{proposition}[Equation 3.3 of \cite{vershynin2018high}] \label{prop:gaussian-norm}
	Let $t > 0$. 
	We have 
	\[
		  \Prx_{\bx\sim N(0,I_n)}\sbra{\abs{\|\bx\| - \sqrt{n}} \geq t} \leq 2e^{-C t^2}\,,
	\]
    where $C$ is an absolute constant. 
\end{proposition}

Let $\S^{n-1} := \{ x : \|x\| = 1\}$ denote the unit sphere in $\R^n$ endowed with the uniform measure. 
It is known (see, for example, Appendix~D of~\cite{KOS:08}) that for any $x \in \R^n$ and $r \geq 0$, we have 
\begin{equation} \label{eq:spherical-cap-exact}
	\Prx_{\bv\sim\S^{n-1}}\sbra{x\cdot \bv \leq r} = \tau_n\int_{-1}^{\min\cbra{\frac{r}{\|x\|}, 1}} (1-z^2)^{\frac{(n-3)}{2}}\,dz \qquad\text{where}~\tau_n := \frac{\Gamma\pbra{\frac{n}{2}}}{\sqrt{\pi}\cdot\Gamma\pbra{\frac{n-1}{2}}}\,.
\end{equation}
Stirling's approximation (or alternatively, Gautschi's inequality~\cite{gautschi1959some}) yields the following: 
\begin{equation} \label{eq:tau-asymp}
	\sqrt{\frac{n}{2\pi}}(1+O(n^{-1})) \leq \tau_n \leq \sqrt{\frac{n}{2\pi}}\,.
\end{equation}

We will also require the following estimate on the measure of spherical caps: 

\begin{lemma} \label{lemma:cap-ub}
	Let $x \in \R^n$ and $r \geq 0$ such that $\|x\| \geq r$. Then 
	\[
		\Prx_{\bv\sim\S^{n-1}}\sbra{x\cdot \bv \leq r} \geq 1 - \tau_n\cdot \frac{\|x\|}{r(n-3)}\cdot e^{-\frac{r^2(n-3)}{2\|x\|^2}}\,.
	\]
\end{lemma}

Standard estimates on the measure of spherical caps---see, for example, \cite{ball1997elementary,tkocz2012upper}---obtain the coarser bound
\[
	\Prx_{\bv\sim\S^{n-1}}\sbra{x\cdot\bv\leq r} \geq 1 - e^{-\frac{nr^2}{2\|x\|^2}}\,.
\]
\Cref{lemma:cap-ub} is sharper when $\|x\| \ll r\sqrt{n}$; this is exactly the regime that will arise later in our proof of~\Cref{thm:naz}. 

\begin{proof}
	Thanks to~\Cref{eq:spherical-cap-exact}, we have
	\begin{align*}
		\Prx_{\bv\sim\S^{n-1}}\sbra{x\cdot \bv > r}
		&= \tau_n \int_{\frac{r}{\|x\|}}^1 (1-z^2)^{\frac{(n-3)}{2}} \,dz \\
		&\leq \tau_n \int_{\frac{r}{\|x\|}}^\infty e^{\frac{-z^2(n-3)}{2}} \,dz \tag{$1+y\leq e^y$} \\
		&= \tau_n\cdot \sqrt{\frac{2\pi}{n-3}}\cdot\Prx_{\by \sim N(0,1)}\sbra{\by > \frac{r\sqrt{n-3}}{\|x\|}} \\
		&\leq \tau_n\cdot \frac{\|x\|}{r(n-3)}\cdot e^{-\frac{r^2(n-3)}{2\|x\|^2}} \tag{\Cref{prop:gaussian-tails}}\,.
	\end{align*}
	The result follows immediately. 
\end{proof}

\subsection{Convex Influence}
\label{subsec:convex-influence}

The following quantity was introduced in~\cite{DNS21itcs,DNS22} as a convex set analogue of the well-studied notion of \emph{total influence} from the analysis of Boolean functions~\cite{odonnell-book}:

\begin{definition}[Total convex influence] \label{def:influence}
	Given a convex set $K \sse \R^n$, we define its \emph{total convex influence} $\TInf[K]$ as 
	\[
		\TInf[K] := 
		\frac{d}{dt}\Vol\pbra{(1+t)K}\bigg|_{t=0}\,.
	\]
\end{definition}

We will also make use of the following alternative characterization of total convex influence from~\cite{DNS23-polytope}, and give a self-contained proof for completeness: 

\begin{lemma}[Lemma~67 of~\cite{DNS23-polytope}] \label{lemma:surf-integral}
	Given {a convex set} $K\sse\R^n$, we have 
	\begin{equation*}
		\TInf[K] = \int_{\partial K} (x\cdot\nu_x)\,\phi(x)\,d\sigma(x)
	\end{equation*}
	where $\nu_x$ denotes the unit normal to $\partial K$ at $x$ and $d\sigma(\cdot)$ denotes the surface measure.
\end{lemma} 

\begin{proof}
    We have 
    \begin{align}
        \TInf[K] 
        = \frac{d}{dt}\pbra{\int_{(1+t)K} \phi(x) \,dx}\bigg|_{t = 0} 
        &= \frac{d}{dt}\pbra{\int_{K} (1+t)^n \, \phi((1+t)y) \,dy}\bigg|_{t = 0} \tag{Writing $x = (1+t)y$} \nonumber \\ 
        &= \int_{K} \pbra{\frac{d}{dt}(1+t)^n \, \phi((1+t)y)}\bigg|_{t=0} \,dy  \tag{Fubini}\\
        &= \int_K (n-\|y\|^2)\,\phi(y)\,dy \label{eq:spectral-def} \\
        &= \int_K \calL\frac{\|y\|^2}{2}\,\phi(y)\,dy\,, \nonumber
    \end{align}
    where, for a function $f:\R^n \to \R$, we define $\calL f(x):= \Delta f(x) - \nabla f(x)\cdot x$. 
    Note that $\calL$ is the Ornstein--Uhlenbeck operator (cf.~Definition~11.24 and Proposition~11.26 of~\cite{odonnell-book}).
    Integrating by parts then gives  
    \[
        \TInf[K] = \int_{\partial K} (x\cdot \nu_x) \, \phi(x)\, d\sigma(x)\,.
    \]
    as desired. 
\end{proof} 

Note that the proof of~\Cref{lemma:surf-integral} does not require convexity of $K$ and in fact applies more generally to any set with finite Gaussian surface area. 
We additionally note that the expression in~\Cref{eq:spectral-def} is the original definition of convex influence from~\cite{DNS21itcs,DNS22}; we record a few easy consequences of it that rely on Hermite analysis (see~\Cref{sec:hermite} or Chapter~11 of~\cite{odonnell-book}). 
\Cref{eq:spectral-def} can be rewritten as 
\[
    \TInf[K] = \Ex_{\bx\sim N(0,I_n)}\sbra{K(\bx)(n-\|\bx\|)^2}\,,
\]
where we identify $K \sse \R^n$ with its $0/1$-valued indicator function. 
Writing $e_i$ for the $i^{\text{th}}$ standard basis vector, we have 
\[
    \TInf[K] = -\sqrt{2}\cdot\sumi \wh{K}(2e_i)\,,
\]
where $\{\wh{K}(2e_i)\}_i$ are the degree-$2$ Hermite coefficients of $K$ (see~\Cref{sec:hermite} and~\Cref{eq:h2}).  
Applying the Cauchy--Schwarz inequality yields  
\begin{equation} \label{eq:influence-ub}
    \TInf[K] \leq \pbra{2n\sumi \wh{K}(2e_i)^2}^{1/2}\,.
\end{equation}
In particular, for all (measurable) $K\sse\R^n$, we have $\TInf[K] \leq \sqrt{2n}$ by Parseval's formula (\Cref{eq:parseval}) since $K:\R^n\to\zo$ is Boolean valued. 

\section{A Simple Proof of~\Cref{thm:naz}}
\label{sec:proof}

As in Nazarov's original argument~\cite{Nazarov:03}, we will show the existence of a convex set with large Gaussian perimeter via the probabilistic method. 
We first give a sketch of why (a slight modification of) Nazarov's construction has $\Omega(n^{1/4})$ Gaussian surface area in~\Cref{subsec:warmup}, and then give the full argument recovering Nazarov's constant in~\Cref{subsec:e-5-4}. 

\subsection{Intuition: An $\Omega(n^{1/4})$ Lower Bound on Gaussian Surface Area}
\label{subsec:warmup}

Recent work~\cite{DNS23-polytope} on polyhedral approximation under the Gaussian measure gives a quick way to see the existence of a convex set with maximal Gaussian perimeter (up to constant factors).
We briefly sketch this argument below, leaving full details to the interested reader. 
The aim here is to give intuitive justification for the more explicit calculation carried out in \Cref{subsec:e-5-4}. 

Let $w, s > 0$ be parameters that we will set later, and let $\Naz'(w,s)$ denote the distribution over convex subsets of $\R^n$ where a draw $\bK\sim\Naz'(w,s)$ is generated as follows: 
\begin{enumerate}
    \item Draw $\g{1}, \dots \g{s} \sim N(0,I_n)$ and set 
    \[
        \bH'_i := \{x \in \R^n : x\cdot\g{i} \leq w\}~\text{for}~i\in[s]\,.
    \]
    \item Output $\bK \sim \cap_i \bH'_i$. 
\end{enumerate}
This distribution---a slight modification of Nazarov's construction which we will use in~\Cref{subsec:e-5-4}---was shown by De, Nadimpalli, and Servedio~\cite{DNS23-polytope} to well-approximate the $\ell_2$-ball of Gaussian measure $1/2$ for appropriate $w$ and $s$. 
In particular, it follows from Theorems~27 and 70 of~\cite{DNS23-polytope} that for $w = \Theta(n^{3/4})$ and $s = \exp(\Theta(\sqrt{n})$, 
\begin{equation} \label{eq:ortlieb}
    \Ex_{\bK\sim \Naz'(w,s)}\big[\TInf[\bK]\big] \geq \Omega(\sqrt{n})\,. 
\end{equation}
To be precise, the above fact relies on an observation made during the proof of Theorem~70 of \cite{DNS23-polytope}: see the paragraph preceding the final displayed equation in its proof. 
It is also possible to directly obtain~\Cref{eq:ortlieb} by computing \smash{$\E[\TInf[\bK]]$} via \Cref{eq:spectral-def}.\footnote{Indeed, note that it follows from~\Cref{eq:spectral-def} and Fubini's theorem that 
\[
    \Ex_{\bK\sim\Naz'(w,s)}\big[\TInf[\bK]\big] = \Ex_{\bx\sim N(0,I_n)}\sbra{\Phi\pbra{\frac{w}{\|\bx\|}}^s\pbra{n-\|\bx\|}^2}\,,
\]
and this can be readily estimated using standard tail bounds for univariate Gaussian random variables (cf.~\Cref{prop:gaussian-tails}) and chi-squared random variables~\cite{laurent2000}. 
}

Next, we will use~\Cref{lemma:surf-integral} to obtain an upper bound on the L.H.S.~of \Cref{eq:ortlieb}. 
We say that $K$ in the support of $\Naz'(r,s)$ is \emph{good} if, for some (large) absolute constant $c'$ $\|g^{(i)}\| \leq c'\sqrt{n}$ for all $i \in [s]$, and will say that it is \emph{bad} otherwise. 
Note that \Cref{prop:gaussian-norm} and a union bound together imply that $\Pr[\bK~\text{is good}] \geq 1-e^{-C'n}$ for some absolute constant $C'$. 
We thus have 
\begin{align*}
    \Ex_{\bK\sim\Naz'(w,s)}\big[\TInf[\bK]\big] 
    &\leq \Ex_{\bK\sim\Naz'(w,s)}\big[\TInf[\bK] \big|\bK~\text{is good}\big]  + e^{-C'n}\Ex_{\bK\sim\Naz'(w,s)}\big[\TInf[\bK]\big|\bK~\text{is bad}\big]  \\ 
    &\leq \Ex_{\bK\sim\Naz'(w,s)}\big[\TInf[\bK] \big|\bK~\text{is good}\big] + e^{-\Theta(n)} \\
    &\leq \Theta(n^{1/4})\cdot\Ex_{\bK\sim\Naz'(w,s)}\sbra{\GSA(\bK)\big|\bK~\text{is good}} + e^{-\Theta(n)}\,, \tag{\Cref{lemma:surf-integral}}
\end{align*}
where the second inequality relied on the fact that $\TInf[L] \leq O(\sqrt{n})$ for all $L\sse\R^n$ (cf.~the discussion following~\Cref{eq:influence-ub}), and the final inequality also relied on the choice of $w = \Theta(n^{3/4})$. 
Together with~\Cref{eq:ortlieb}, this immediately gives the existence of a good $K$ with $\GSA(K) \geq \Omega(n^{1/4})$. 

\subsection{Recovering Nazarov's Constant}
\label{subsec:e-5-4}

We now refine the argument sketched in~\Cref{subsec:warmup} to prove~\Cref{thm:naz}. 
Let $r := n^{1/4}$ and $s > 0$ be a parameter that we will set later. 
(We write $r$ instead of $w$ to distinguish this section's construction from that of the previous one.)
Let $\Naz(r,s)$ denote the distribution over convex subsets of $\R^n$ where a draw $\bK \sim \Naz(r,s)$ is generated as follows: 
\begin{enumerate}
	\item Draw $\vv{1}, \dots, \vv{s} \sim \S^{n-1}$ and set 
		\[
			\bH_i := \{ x \in \R^n : x\cdot\vv{i} \leq r \}~\text{for}~i\in[s]\,.
		\]
	\item Output $\bK := \bigcap_{i} \bH_i$. 
\end{enumerate}

Note that for every $K$ in the support of $\Naz(r,s)$, we have $(x\cdot\nu_x) = r$ for all $x \in \partial K$. 
The following is then immediate from~\Cref{lemma:surf-integral}:
\begin{equation} \label{eq:naz-SA-inf}
	\Ex_{\bK \sim \Naz(r,s)}\big[\TInf[\bK]\big] = r \cdot \Ex_{\bK\sim\Naz(r,s)}\sbra{\GSA(\bK)}\,.
\end{equation}

On the other hand, computing $\TInf[\bK]$ using \Cref{def:influence} gives  
\begin{align*}
	\Ex_{\bK \sim \Naz(r,s)}\big[\TInf[\bK]\big] 
	&= \Ex_{\bK\sim\Naz(r,s)}\sbra{\frac{d}{dt}\Vol((1+t)
	\bK)\bigg|_{t=0}} \\
	&= \Ex_{\bK\sim\Naz(r,s)}\sbra{\frac{d}{dt}\Prx_{\bx\sim N(0,I_n)}\sbra{\bx\in(1+t)
	\bK}\bigg|_{t=0}} \\
	&= \Ex_{\bx\sim N(0,I_n)}\sbra{\frac{d}{dt}\Prx_{\bK\sim \Naz(r,s)}\sbra{\bx\in (1+t)\bK}\bigg|_{t=0}} \tag{Fubini}\\
	&= \Ex_{\bx\sim N(0,I_n)}\sbra{\frac{d}{dt}\pbra{\Prx_{\bv\sim\S^{n-1}}\sbra{\bx\cdot\bv\leq (1+t)r}}^s\bigg|_{t=0}}\,.
\end{align*}
Applying the chain rule, we get 
\[
	\frac{d}{dt}\pbra{\Prx_{\bv\sim\S^{n-1}}\sbra{x\cdot\bv\leq (1+t)r}}^s\bigg|_{t=0} 
	= 
	s\cdot \Prx_{\bv\sim\S^{n-1}}\sbra{x\cdot\bv\leq r}^{s-1}\cdot\frac{d}{dt}\pbra{\Prx_{\bv\sim\S^{n-1}}\sbra{x\cdot\bv\leq r(1+t)}}\bigg|_{t=0}\,.
\]
It is readily verified using the Leibniz rule that 
\[
	\frac{d}{dt}\pbra{\Prx_{\bv\sim\S^{n-1}}\sbra{x\cdot\bv\leq r(1+t)}}\bigg|_{t=0} = \tau_n\cdot\frac{r}{\|x\|}\cdot\pbra{1 - \frac{r^2}{\|x\|^2}}_{+}^{(n-3)/2}\,,
\]
where $(a)_+ := \max\{a, 0\}$. 
We thus have 
\begin{equation} \label{eq:inf-dilation-equality}
	\Ex_{\bK \sim \Naz(r,s)}\big[\TInf[\bK]\big] = \Ex_{\bx\sim N(0,I_n)}\sbra{s\cdot\frac{r}{\|\bx\|}\cdot\tau_n\cdot\pbra{1-\frac{r^2}{\|\bx\|^2}}^{(n-3)/2}_{+}\cdot\Prx_{\bv\sim\S^{n-1}}\sbra{\bx\cdot\bv\leq r}^{s-1}}\,.
\end{equation}

The remainder of the argument will establish 
\begin{equation} \label{eq:broadsheet}
	\text{R.H.S.~of \eqref{eq:inf-dilation-equality}} \geq e^{-5/4}\cdot\sqrt{n}(1-o(1))\,,
\end{equation}
which, together with~\Cref{eq:naz-SA-inf}, completes the proof of~\Cref{thm:naz}. 
Let $t = n^{1/4}$, and let $A := \{x : \|x\| \in [\sqrt{n} - t, \sqrt{n} + t] \}$.\footnote{We note that any choice of $t$ that is $\omega(1)$ but $o(n^{1/2})$ will do.} 
By~\Cref{prop:gaussian-norm}, we have $\Vol(A) \geq 1 - 2\exp(-Ct^2) = 1-o(1)$. 
Since the quantity inside the expectation in~\Cref{eq:inf-dilation-equality} is non-negative, we have 
\begin{equation} \label{eq:nitro-shandy}
	\text{R.H.S.~of \eqref{eq:inf-dilation-equality}} \geq \pbra{1 - o(1)}\cdot\inf_{x\in A} \cbra{s\cdot\frac{r}{\|x\|}\cdot\tau_n\cdot\pbra{1-\frac{r^2}{\|x\|^2}}^{(n-3)/2}_{+}\cdot\Prx_{\bv\sim\S^{n-1}}\sbra{x\cdot\bv\leq r}^{s-1}}\,. 
\end{equation}

Let $F : \R^n \to \R$ be the function 
\[
    F(x) = \frac{r}{\|x\|}\cdot\pbra{1-\frac{r^2}{\|x\|^2}}^{(n-3)/2}_{+}\,.
\]
From~\Cref{lemma:cap-ub}, we get
\[
    \Prx_{\bv\sim\S^{n-1}}\sbra{x\cdot\bv \leq r} \geq 1 - \frac{\tau_n\|x\|}{r(n-3)}\exp\pbra{-\frac{r^2(n-3)}{2\|x\|^2}} =: 1- G(x)\,. 
\]
We take $s$ to satisfy $s\cdot\inf_{x \in A} \cbra{F(x)} = c_1$ where $c_1$ is a parameter we will set later. 
To control $1-s\cdot G(x)$, note that  
\[
    \frac{F(x)}{G(x)} = \frac{r^2 (n-3)}{\tau_n\cdot \|x\|^2}\cdot\pbra{1-\frac{r^2}{\|x\|^2}}^{(n-3)/2}_{+} \exp\pbra{\frac{r^2(n-3)}{2\|x\|^2}}\,.
\]
Recall that $\ln(1-a) \geq -a - a^2/2 - C'a^3$ for some absolute constant $C'$ (this is readily seen by taking a Taylor expansion).
Consequently, for $x \in A$ we have 
\[
    \frac{F(x)}{G(x)} \geq (1-o(1))\frac{\sqrt{n}}{\tau_n}\exp\pbra{-\frac{(n-3)}{2}\pbra{\frac{r^4}{2\|x\|^4} + C'\frac{r^6}{\|x\|^6}}} \geq \frac{\sqrt{2\pi}}{e^{1/4}}(1-o(1))\,,
\]
where the second inequality relies on \Cref{eq:tau-asymp} and our choice of $r = n^{1/4}$. 
In particular, 
\begin{equation} \label{eq:maine}
    \inf_{x\in A} F(x) \geq \frac{\sqrt{2\pi}}{e^{1/4}}(1-o(1))\cdot\inf_{x\in A} G(x)
\end{equation}
Finally, it is readily verified that for $x \in A$ we have 
\begin{equation} \label{eq:mass}
    \frac{G(x)}{\inf_{x\in A} G(x)} = 1+o(1)\,. 
\end{equation} 

Returning to~\Cref{eq:nitro-shandy}, we can rewrite it as
\begin{align*}
    \text{R.H.S.~of~\eqref{eq:inf-dilation-equality}} 
    &\geq (1-o(1))\cdot \tau_n \cdot \inf_{x\in A}\cbra{s\cdot F(x)\cdot (1-G(x))^{s-1}}\\
    &\geq (1-o(1))\cdot \tau_n \cdot c_1 \cdot \inf_{x\in A} \cbra{(1-G(x))^s}\,. 
\end{align*}
Using the inequality $1-a \geq e^{-a}(1-a^2 e^a)$ for $a \geq 0$, we get
\begin{align}
    \text{R.H.S.~of~\eqref{eq:inf-dilation-equality}} 
    &\geq (1-o(1))\cdot \tau_n \cdot c_1 \cdot \inf_{x\in A} \cbra{\pbra{e^{-G(x)}\pbra{1-G(x)^2e^{G(x)}}}^s}  \nonumber \\ 
    &\geq  (1-o(1))\cdot \tau_n \cdot c_1 \cdot \inf_{x\in A} \cbra{e^{-s G(x)} \pbra{1 - s G(x)^2e^{G(x)}}}  \label{eq:stitch} \\
    &\geq (1-o(1))\cdot \tau_n \cdot c_1 \exp\pbra{-\frac{c_1 e^{1/4}}{\sqrt{2\pi}}}\,,\label{eq:roadrunner}
\end{align}
where \Cref{eq:stitch} uses Bernoulli's inequality, and \Cref{eq:roadrunner} relies on (a) \Cref{eq:maine,eq:mass}, and (b) the fact that $1 - sG(x)^2e^{G(x)} = 1-o(1)$ for $x\in A$. 
Finally, \Cref{eq:tau-asymp} implies that 
\[
    \text{R.H.S.~of~\eqref{eq:inf-dilation-equality}} \geq (1-o(1))\cdot\sqrt{n}\cdot \frac{c_1}{\sqrt{2\pi}}\exp\pbra{-\frac{c_1e^{1/4}}{\sqrt{2\pi}}}\,.
\]
Optimizing with respect to $c_1$ then yields \Cref{eq:broadsheet}, completing the proof. 
\qed 

\bigskip

One can check that modifying 
$r$ by a constant factor does not lead to a better constant than that obtained by~\Cref{thm:naz}. We leave this verification to the interested reader. 

\subsection{Discussion: Upper Bounds on GSA via Convex Influence}
\label{sec:GSA-ubs}

It is natural to ask whether the notion of convex influence can also be used to derive effective \emph{upper} bounds on Gaussian surface area.  
For simplicity, we assume throughout that $K \subseteq \R^n$ is convex and contains the origin. 
Next, let $\rinn(K)$ denote the radius of the largest origin-centered ball contained in $K$. 
Since $0^n \in K$, we have $\rinn(K) \geq 0$. 
It follows from the definition of $\TInf[K]$ (cf.~\Cref{def:influence}) that 
\[
    \TInf[K] 
    = \lim_{t\to 0}\frac{\Vol((1+t)K \setminus K)}{t} 
    \geq \lim_{t\to 0}\frac{\Vol(K_{t\rinn(K)} \setminus K)}{t} 
    = \rinn(K)\cdot\GSA(K)\,.
\]
Combining this with~\Cref{eq:influence-ub} gives 
\begin{equation} \label{eq:final-deg2}
    \GSA(K) \leq \frac{1}{\rinn(K)}\cdot\pbra{2n\sumi \wh{K}(2e_i)^2}^{1/2}\,.
\end{equation}
Using Parseval's identity (see~\Cref{eq:fourier-formulas}), we can obtain a more convenient expression: 
\begin{equation} \label{eq:final-var}
    \GSA(K) \leq \frac{\sqrt{2n}\cdot\sqrt{\Var[K]}}{\rinn(K)}\,.
\end{equation}
\Cref{eq:final-var} is particularly effective when $\Vol(K)$ is large, as then $\Var[K]$ is close to zero and $\rinn(K) \geq \Omega(1)$ by convexity. 

One might optimistically hope that the right-hand side of~\Cref{eq:final-var} could be shown to be $O(n^{1/4})$ in general, yielding an alternative proof of Ball’s bound (\Cref{thm:ball}). 
Unfortunately, this is false even in simple cases: for example, in one dimension, the bound fails for a thin slab $S=\{x:|x|\leq \theta\}$ as $\theta\to 0$. 
It remains possible, however, that the refined bound from \Cref{eq:final-deg2} or~\Cref{lemma:surf-integral} itself could furnish a new proof of Ball's theorem. 
We leave this as an open direction for future work.

\section*{Acknowledgements}
\label{sec:acks}

C.P.~is funded by the MIT Undergraduate Research Opportunities Program (UROP). 
We thank Elchanan Mossel and Rocco Servedio for advice and encouragement to write this note, and Amit Rajaraman for helpful comments on an earlier version. 

\bibliography{allrefs.bib}
\bibliographystyle{alphaurl}

\appendix

\section{Hermite Analysis}
\label{sec:hermite}

Our notation and terminology follow Chapter~11 of ~\cite{odonnell-book}. 
For $n \in \N_{>0}$, we write $L^2(\R^n, N(0,I_n))$ to denote the space of functions $f: \R^n \to \R$ that have finite second moment under the Gaussian distribution, i.e. $f\in L^2$ if 
\[
    \|f\|^2 = \Ex_{\bz \sim N(0,I_n)} \left[f(\bz)^2\right]^{1/2} < \infty\,.
\]
When the dimension $n$ is clear from context, we will write $L^2 = L^2(\R^n, N(0,I_n))$ for simplicity. 
We view $L^2$ as an inner product space with 
\[
    \la f, g \ra := \Ex_{\bx \sim  N(0,I_n)}[f(\bx)g(\bx)]\,.
\]

We recall the Hermite basis for $L^2$:

\begin{definition}[Hermite basis]
	The \emph{Hermite polynomials} $(h_j)_{j\in\N}$ are the univariate polynomials defined as
	$$h_j(x) = \frac{(-1)^j}{\sqrt{j!}} \exp\left(\frac{x^2}{2}\right) \cdot \frac{d^j}{d x^j} \exp\left(-\frac{x^2}{2}\right).$$
\end{definition}

In particular, one can check that 
\begin{equation} \label{eq:h2}
    h_2(x) = \frac{x^2-1}{\sqrt{2}}\,.
\end{equation}
The following fact is standard:

\begin{fact} [Proposition~11.33 of~\cite{odonnell-book}] \label{fact:hermite-orthonormality}
	The Hermite polynomials $(h_j)_{j\in\N}$ form a complete, orthonormal basis for $L^2(\R, N(0,1))$. For $n > 1$ the collection of $n$-variate polynomials given by $(h_\alpha)_{\alpha\in\N^n}$ where
	$$h_\alpha(x) := \prod_{i=1}^n h_{\alpha_i}(x)$$
	forms a complete, orthonormal basis for $L^2$. 
\end{fact}

Given a function $f \in L^2$ and $\alpha \in \N^n$, we define its \emph{Hermite coefficient on} $\alpha$ as $\wh{f}(\alpha) := \la f, h_\alpha \ra$. It follows that $f:\R^n\to\R$ can be uniquely expressed as 
\[f = \sum_{\alpha\in\N^n} \wh{f}(\alpha)h_\alpha\]
with equality holding in $L^2$. 
One can check that Parseval's identity holds in this setting:
\begin{equation} \label{eq:parseval}
    \abra{f, g} = \sum_{\alpha\in\N^n}\wh{f}(\alpha) \wh{g}(\alpha)\,. 
\end{equation}
It is also readily verified that for $f \in L^2$, 
\begin{equation} \label{eq:fourier-formulas}
    \Ex_{\bx\sim N(0,I_n)}[f(\bx)] = \wh{f}(0^n) 
    \qquad\text{and}\qquad 
    \Varx_{\bx\sim N(0,I_n)}\sbra{f(\bx)} = \sum_{\alpha\neq 0^n} \wh{f}(\alpha)^2\,.
\end{equation} 
\end{document}